\documentclass[11pt,oneside]{amsart}

\usepackage[a4paper, total={6in, 8in}]{geometry}

\usepackage{setspace}

\usepackage{amssymb}   

\usepackage{amsmath}

\usepackage{mathrsfs}

\usepackage{stmaryrd}

\usepackage{amsthm}
\usepackage{newlfont}
\usepackage{amscd}
\usepackage{mathtools}
\usepackage{bm}
\usepackage{tikz-cd}
\usepackage{graphicx}

\usepackage{hyperref}
\usepackage{framed}
\usepackage{comment}

\theoremstyle{definition}
\newtheorem{definition}{Definition}[section]

\theoremstyle{remark}
\newtheorem{remark}[definition]{Remark}

\theoremstyle{theorem}
\newtheorem{theorem}[definition]{Theorem}

\theoremstyle{corollary}

\theoremstyle{lemma}
\newtheorem{lemma}[definition]{Lemma}

\theoremstyle{example}
\newtheorem{example}[definition]{Example}

\theoremstyle{prop}
\newtheorem{prop}[definition]{Proposition}

\usepackage{enumitem}

\newcommand {\coker} {\operatorname{coker}}

\newcommand {\Coh} {\mathrm{Coh}}
\newcommand {\rank} {\operatorname{rank}}

\title{Modular variants of $p$-adic fundamental sequence}

\begin{document}
\author{Heng Du} 
\address{Yau Mathematical Sciences Center, Tsinghua University, Beijing 100084, China}
\email{hengdu@mail.tsinghua.edu.cn}
\author{Qingyuan Jiang} 
\address{Department of Mathematics,
The Hong Kong University of Science and Technology, Clearwater Bay, Kowloon, Hong Kong.} 
\email{jiangqy@ust.hk}
\author{Yucheng Liu} 
\address{College of Mathematics and Statistics, Center of Mathematics, Chongqing University, Chongqing, 401331, China}
\email{noahliu@cqu.edu.cn}
\begin{abstract} 
In this article, we relate any Farey triangle in the extended upper half-plane to a variant of Colmez--Fontaine's fundamental lemma in $p$-adic Hodge theory. In particular, their original fundamental lemma corresponds to the fundamental Farey triangle $(\frac{1}{0},\frac{1}{1},\frac{0}{1})$. 
\end{abstract}

\maketitle


\section{Introduction}
In $p$-adic Hodge theory, one of the most foundational results is that every weakly admissible filtered $(\phi, N)$-module is admissible. It implies that the $D_{\mathrm{st}}$ functor for semistable Galois representations is fully faithful. In other words, the linear-algebraic data of filtered $(\phi, N)$-modules are enough to characterize crystalline and semistable Galois representations. 

Several proofs of this foundational result exist in the literature (see \cite{constructionofsemistable, BanachColmezspaces, p-adicdifferentialequations, CrystallinerepresentationsandF-crystals, farguesfontaine-courbes}). In their original approach, one of the key ingredients in the proof of Colmez and Fontaine is the \emph{fundamental lemma in $p$-adic Hodge theory} (\cite[Proposition 2.1]{constructionofsemistable}). Later, this fundamental lemma was elegantly refined by Colmez using the language of Banach--Colmez spaces (\cite[Proposition 7.11]{BanachColmezspaces})\footnote{See also \cite[Lemma 2.4.1]{Analyticp-adicBanachspaces} for another version of the fundamental lemma.}. 

Recall that in \cite{BanachColmezspaces} (where they are termed \textit{Espaces Vectoriels de dimension finie}), a Banach--Colmez space is a functor on the category of sympathetic $\mathbb{C}_p$-algebras. Let $\mathbb{V}^1$ denote the Banach--Colmez space which evaluates to $\Lambda$ on any sympathetic algebra $\Lambda$ (see \cite[Example 7.1(ii)]{BanachColmezspaces}). Colmez's refined fundamental lemma can be stated as follows:

\begin{theorem}[{\cite[Proposition 7.11]{BanachColmezspaces}\footnote{In \cite[Proposition~2.1]{constructionofsemistable}, only a weak form of this result is used.}}]\label{thm: colmez intro}
   Let \(\mathbb{W}\) be a Banach--Colmez space admitting a presentation of the form
\[
0 \longrightarrow U \longrightarrow \mathbb{W} \xrightarrow{\,\,\alpha\,\,} \mathbb{V}^1
\longrightarrow 0,
\]
where \(U\) is a finite-dimensional \(\mathbb{Q}_p\)-Banach--Colmez Space. Let \(\psi: \mathbb{W} \longrightarrow \mathbb{V}^1\)
be a morphism of Banach--Colmez spaces. Then one of the following two alternatives holds.
\begin{enumerate}
    \item \(\psi\) is surjective, and \(\ker \psi\) is a finite-dimensional
\(\mathbb{Q}_p\)-Banach--Colmez Space whose \(\mathbb{Q}_p\)-dimension is $\dim_{\mathbb{Q}_p} U$;

\item \(\operatorname{Im}\psi\) is a finite-dimensional \(\mathbb{Q}_p\)-Banach--Colmez Space, and
\[
0 \longrightarrow U\cap \ker\psi
\longrightarrow \ker\psi
\xrightarrow{\,\,\alpha\,\,} \mathbb{V}^1
\longrightarrow 0
\]
is a presentation of \(\ker\psi\), with
\[
\dim_{\mathbb{Q}_p}(U\cap \ker\psi)
=
\dim_{\mathbb{Q}_p}U-\dim_{\mathbb{Q}_p}(\operatorname{Im}\psi).
\]
\end{enumerate}

\end{theorem}


A key observation of this paper is that Theorem~\ref{thm: colmez intro} admits a natural interpretation using the Fargues--Fontaine curve. Through Le Bras' equivalence \cite{LeBrasresult}, the category of Banach--Colmez spaces is equivalent to a tilted heart in the derived category of coherent sheaves on the absolute Fargues--Fontaine curve. Under this perspective, together with the Harder--Narasimhan theory on the absolute Fargues--Fontaine curve, the objects appearing in the fundamental lemma correspond to stable coherent sheaves of slopes $0=\frac{0}{1}$, $1=\frac{1}{1}$, and $\infty=\frac{1}{0}$. On the extended upper half-plane, these rational numbers precisely form the vertices of the fundamental Farey triangle $(\frac{1}{0},\frac{1}{1},\frac{0}{1})$.


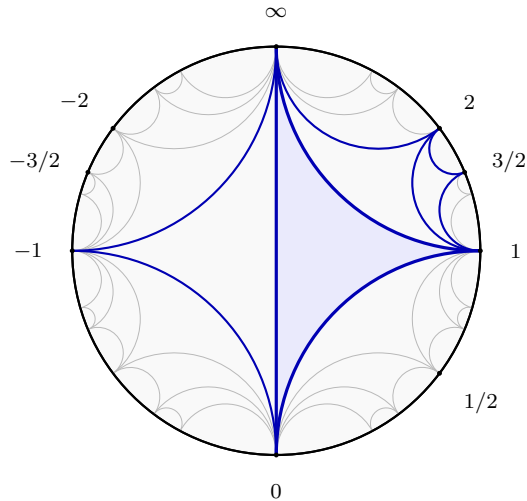
\begin{figure}[htbp]
\centering
\begin{tikzpicture}[scale=2.7, line cap=round, line join=round]

\newcommand{\geodesic}[5][]{%
  \pgfmathsetmacro{\xa}{2*(#2)*(#3)/((#2)*(#2)+(#3)*(#3))}%
  \pgfmathsetmacro{\ya}{((#2)*(#2)-(#3)*(#3))/((#2)*(#2)+(#3)*(#3))}%
  \pgfmathsetmacro{\xb}{2*(#4)*(#5)/((#4)*(#4)+(#5)*(#5))}%
  \pgfmathsetmacro{\yb}{((#4)*(#4)-(#5)*(#5))/((#4)*(#4)+(#5)*(#5))}%
  \pgfmathsetmacro{\dett}{\xa*\yb-\ya*\xb}%
  \pgfmathsetmacro{\absdet}{abs(\dett)}%
  \ifdim \absdet pt < 0.0001pt
    \draw[#1] (\xa,\ya) -- (\xb,\yb);
  \else
    \pgfmathsetmacro{\cx}{(\yb-\ya)/\dett}%
    \pgfmathsetmacro{\cy}{(\xa-\xb)/\dett}%
    \pgfmathsetmacro{\rr}{sqrt(max(0,\cx*\cx+\cy*\cy-1))}%
    \pgfmathsetmacro{\anga}{atan2(\ya-\cy,\xa-\cx)}%
    \pgfmathsetmacro{\angb}{atan2(\yb-\cy,\xb-\cx)}%
    \pgfmathsetmacro{\delta}{\angb-\anga}%
    \ifdim \delta pt > 180pt
      \pgfmathsetmacro{\angb}{\angb-360}%
    \fi
    \ifdim \delta pt < -180pt
      \pgfmathsetmacro{\angb}{\angb+360}%
    \fi
    \draw[#1] (\xa,\ya) arc[start angle=\anga, end angle=\angb, radius=\rr];
  \fi
}

\newcommand{\fraclabel}[4][]{%
  \pgfmathsetmacro{\x}{2*(#2)*(#3)/((#2)*(#2)+(#3)*(#3))}%
  \pgfmathsetmacro{\y}{((#2)*(#2)-(#3)*(#3))/((#2)*(#2)+(#3)*(#3))}%
  \fill (\x,\y) circle (0.012);
  \node[#1, font=\scriptsize, inner sep=1pt]
    at ({1.13*\x},{1.13*\y}) {$#4$};
}

\fill[gray!5] (0,0) circle (1);
\draw[thick] (0,0) circle (1);

\begin{scope}
\clip (0,0) circle (1.002);

\foreach \p in {-4,...,4}{
  \geodesic[gray!55, line width=0.25pt]{\p}{1}{1}{0}
}

\foreach \q in {1,...,4}{
  \foreach \p in {-4,...,4}{
    \foreach \s in {1,...,4}{
      \foreach \r in {-4,...,4}{
        \pgfmathtruncatemacro{\detfarey}{abs(\p*\s-\q*\r)}
        \pgfmathtruncatemacro{\ordered}{(\p*\s < \r*\q) ? 1 : 0}
        \ifnum\detfarey=1
          \ifnum\ordered=1
            \geodesic[gray!55, line width=0.25pt]{\p}{\q}{\r}{\s}
          \fi
        \fi
      }
    }
  }
}

\fill[blue!12, opacity=0.6]
  (0,-1)
  arc[start angle=180, end angle=90, radius=1]
  arc[start angle=-90, end angle=-180, radius=1]
  -- cycle;

\geodesic[blue!70!black, very thick]{0}{1}{1}{1}
\geodesic[blue!70!black, very thick]{1}{1}{1}{0}
\geodesic[blue!70!black, very thick]{0}{1}{1}{0}

\geodesic[blue!70!black, thick]{1}{1}{2}{1}
\geodesic[blue!70!black, thick]{2}{1}{1}{0}
\geodesic[blue!70!black, thick]{1}{1}{3}{2}
\geodesic[blue!70!black, thick]{3}{2}{2}{1}
\geodesic[blue!70!black, thick]{-1}{1}{0}{1}
\geodesic[blue!70!black, thick]{-1}{1}{1}{0}

\end{scope}

\draw[thick] (0,0) circle (1);

\fraclabel[below]{0}{1}{0}
\fraclabel[right]{1}{1}{1}
\fraclabel[above]{1}{0}{\infty}

\fraclabel[above right]{2}{1}{2}
\fraclabel[right]{3}{2}{3/2}
\fraclabel[below right]{1}{2}{1/2}

\fraclabel[left]{-1}{1}{-1}
\fraclabel[above left]{-2}{1}{-2}
\fraclabel[left]{-3}{2}{-3/2}

\end{tikzpicture}

\caption{A finite portion of the Farey tessellation in the Poincar\'e disc.}
\label{fig:farey-poincare-disc}
\end{figure}

Motivated by this insight, we show that this phenomenon is not isolated to the slopes $0, 1$, and $\infty$. Our main result is a generalization of the fundamental lemma to \textit{any} Farey triangle in the extended upper half-plane:

\begin{theorem}[Main theorem, weak version]\label{thm: main in intro}
For $\lambda = h/d \in \mathbb{Q}_\infty$\footnote{We will always assume $\lambda = h/d$ is in its simplest form, that is $(d,h)\in \mathbb{N}_{> 0}\times {\mathbb{Z}}$ with $d=\min\{a\in \mathbb{N}_{>0} \mid a\lambda \in \mathbb{Z}\}$ when $\lambda \neq \infty$. We write $\infty = 1/0$, and $\infty > 0$.} with $\lambda \geq 0$, we define a Banach--Colmez space $\mathbb{U}_\lambda$ as follows: for $\lambda > 0$, let $\mathbb{U}_{\lambda}=(\mathbb{B}_{\mathrm{cris}}^+)^{\varphi^h=p^d}$ be the \emph{effective} Banach--Colmez space $\mathbb{U}_{d,h}$ as in \cite{BanachColmezspaces} (see also Example~\ref{eg:BC:period.rings}), and set $\mathbb{U}_0 = \mathbb{V}^1$. For $\lambda_1, \lambda_2, \lambda_3 \in \mathbb{Q}_\infty$ with $\lambda_i \ge 0$, consider a complex of Banach--Colmez spaces
\begin{equation}\label{eq: fundamental exact seq}
\mathbb{U}_{\lambda_1} \longrightarrow \mathbb{U}_{\lambda_2} \xrightarrow{\,\,\psi\,\,} \mathbb{U}_{\lambda_3}.   
\end{equation}
We call (\ref{eq: fundamental exact seq}) a \emph{fundamental exact sequence} if it satisfies a property similar to that in Theorem~\ref{thm: colmez intro}, that is, if $\psi \neq 0$, then it is surjective, with kernel isomorphic to $\mathbb{U}_{\lambda_1}$. Then (\ref{eq: fundamental exact seq}) is a fundamental exact sequence if $(\lambda_1,\lambda_2,\lambda_3)$ is an ordered Farey triangle (see Definition~\ref{def:Farey-triangle} for the precise definition, and Figure~\ref{fig:farey-poincare-disc} for an illustration).
\end{theorem}

Theorem~\ref{thm: main in intro} is only a weak form of the main result. Later we will state a stronger version, closer in form to Theorem~\ref{thm: colmez intro}; in particular, the effectivity (non-negativity) condition on the slopes will no longer be imposed. See Theorem~\ref{thm:SL2 variants}. 

\begin{remark}
The converse statement of Theorem~\ref{thm: main in intro} also holds: if a sequence of the form
\eqref{eq: fundamental exact seq} is fundamental, then the ordered triple
$(\lambda_1,\lambda_2,\lambda_3)$ is necessarily a Farey triangle; see
Theorem 4.3 of \cite{Continuumenvelops}. We postpone the discussion on this converse to \cite{Continuumenvelops}.
\end{remark}

\begin{remark}
In Theorem~\ref{thm: main in intro}, the parameter $\lambda=h/d$ follows the same convention as the Farey
triangles in Definition~\ref{def:Farey-triangle}. This convention is reciprocal
to the usual slope $d/h$ of the corresponding vector bundle
$\mathcal{O}(d/h)$. It is natural from the point of view of the tilted heart
and is also compatible with the dimension-vector notation for
Banach--Colmez spaces; see Remark~\ref{rmk:Farey-triangle-convention}.
\end{remark}

To establish Theorem~\ref{thm: main in intro}, we rely on two main ingredients. The first is Fargues and Fontaine's classification theorem of vector bundles on Fargues--Fontaine curves \cite{farguesfontaine-courbes}. The second is Le Bras' equivalence \cite{LeBrasresult} that the category of Banach--Colmez spaces is equivalent to a tilted heart in the derived category of coherent sheaves on the absolute Fargues--Fontaine curve.

\begin{remark}
Theorem~\ref{thm: main in intro} and its general form naturally extend to the case of $\mathbb{B}_{\mathrm{cris},E}^+$ where $E$ is a finite unramified field extension of $\mathbb{Q}_p$. This generalization follows directly from the fact that our two primary ingredients remain valid in this setting: Fargues--Fontaine's classification of vector bundles on the curve defined for arbitrary $E$, and Le Bras' equivalence between the category of Banach--Colmez spaces and a tilted heart in the derived category of coherent sheaves. 
\end{remark}

\subsection*{Organization of the paper} In \S 2, we recall basic definitions and results of Banach--Colmez spaces and Fargues--Fontaine curves, including Le Bras' equivalence. We also prove a derived version of Le Bras' equivalence in this section. In \S 3, we prove a variant of the fundamental lemma for any Farey triangle in the extended upper half-plane. 

\subsection*{Notation} We use $\frac{1}{0}$ to denote the infinite point in the extended upper half plane, and we regard it as positive by convention. We use $\mathbb{Q}_{\infty}$ to denote the union of rational numbers $\mathbb{Q}$ and infinity $\frac{1}{0}.$ We always express a rational number in its reduced form with positive denominator.

In this paper, we fix a prime number $p$, and let $\mathbb{Q}_p$ be the $p$-adic completion of $\mathbb{Q}$.  Let $\mathbb{Q}_{p^h}$ be the unique unramified extension of $\mathbb{Q}_p$ with residue field $\mathbb{F}_{p^h}$, and $\Breve{\mathbb{Q}}_p$ be the $p$-adic completion of the maximal unramified extension of $\mathbb{Q}_p$. 

We follow the notations in \cite{farguesfontaine-courbes} for Fontaine's period rings. We use $t=\log([\epsilon])$ to denote the $p$-adic analog of $2\pi i$, where $\epsilon=(1,\zeta_p, \zeta_{p^2}, \cdots)$ is a system of compatible $p^n$-th roots of unity, and $[\epsilon]$ is its Teichm\"uller lift.

\medskip
\noindent
\textbf{Acknowledgments.} The authors would like to thank Sebastian Bartling and Laurent Fargues for many helpful discussions.

\section{Banach--Colmez spaces and Fargues--Fontaine curves}

In this section, we will briefly summarize some basic definitions and facts for Banach--Colmez spaces and Fargues--Fontaine curves. For full details, we refer the reader to the references listed below.
\subsection{Banach--Colmez spaces}
We fix $C$ a complete algebraic closed field extension of $\mathbb{Q}_p$. $C$ is a non-Archimedean field with ring of integers $\mathcal{O}_C$ and residue field $k$. $C$ is a perfectoid field in the sense of \cite[Definition 1.2]{scholze-perfectoidspacesIHES}; we let $C^\flat$ be the tilt. In particular, $C^\flat$ is a non-Archimedean field over $\mathbb{F}_p$ with ring of integers $\mathcal{O}_C^\flat$. We write $A_{\mathrm{inf}}=A_{\mathrm{inf}}(\mathcal{O}_C)=W(\mathcal{O}_C^\flat)$\footnote{For a finite extension \(E\) of \(\mathbb{Q}_{p}\), using the theory of \(\mathcal{O}_{E}\)-Witt vectors, one can similarly define the period rings \(A_{\mathrm{inf},E}\), \(A_{\mathrm{cris},E}\), etc.} be the infinitesimal period ring of Fontaine. There is a canonical surjection $\vartheta\colon A_{\mathrm{inf}} \to \mathcal{O}_C$ such that $\ker(\vartheta)$ is principally generated by $\xi \in A_{\mathrm{inf}}$. Denote $A_{\mathrm{cris}}$ as the $p$-complete PD envelope of $A_{\mathrm{inf}}$ with respect to $\ker(\vartheta)$, and define $B_{\mathrm{cris}}^+=A_{\mathrm{cris}}[1/p]$. The map $\vartheta$ extends to a map $\vartheta: B_{\mathrm{cris}}^+ \to C$. We define $B_{\mathrm{dR}}^+$ to be the $(\xi)$-adic completion of $A_{\mathrm{inf}}[1/p]$. Let $\Breve{\mathbb{Q}}_p=W(\overline{\mathbb{F}}_p)[1/p]$, which carries a Frobenius endomorphism $\varphi$ lifting the $p$-th power map modulo $p$. Since $k$ is algebraically closed, there is a canonical inclusion $\overline{\mathbb{F}}_p \to k$, which induces a map $\Breve{\mathbb{Q}}_p \to B_{\mathrm{cris}}^+$. 

Let $K = C$ or $C^\flat$. We denote the big pro-\'{e}tale site of $K$ as $\mathrm{Perf}_{K, \text{pro\'{e}t}}$ and the $v$-site of $K$ as $\mathrm{Perf}_{K, v}$ (we refer to \cite{scholze2017etale} or \cite[\S 2.1]{LeBrasresult} for the definitions of pro-\'{e}tale and $v$-topology). A theorem of Scholze states that there are canonical equivalences $\mathrm{Perf}_{C, \text{pro\'{e}t}} \simeq \mathrm{Perf}_{C^\flat, \text{pro\'{e}t}}$ and $\mathrm{Perf}_{C, v} \simeq \mathrm{Perf}_{C^\flat, v}$ (see \cite{scholze2017etale}  or \cite[Theorem 2.7]{LeBrasresult}). Let $\mathcal{S}\mathrm{hv}_{\text{pro\'{e}t}}(C; \mathbb{Q}_p)$ (resp. $\mathcal{S}\mathrm{hv}_{v}(C; \mathbb{Q}_p)$) denote the abelian category of sheaves of $\mathbb{Q}_p$-vector spaces on $\mathrm{Perf}_{C,\text{pro\'{e}t}}$ (resp. on $\mathrm{Perf}_{C,v}$) and let $D_{\text{pro\'{e}t}}(C; \mathbb{Q}_p)$ (resp. $D_{v}(C; \mathbb{Q}_p)$) denote its $\infty$-categorical derived category.

\begin{example}\label{example.BC.Q_p.C:v-sheaves}
\begin{enumerate}
	\item The constant presheaf $\underline{\mathbb{Q}_{p^h}}$ is a sheaf for the $v$-topology (and hence also for the pro-{\'e}tale topology); see \cite[Proposition 2.9]{LeBrasresult}.
	\item The presheaf $\mathcal{O}_C$ which associates each $S \in \mathrm{Perf}_{C}$ to $\mathcal{O}_S(S)$ is a sheaf for the $v$-topology (and hence also for the pro-{\'e}tale topology); see \cite[Theorem 8.7]{scholze2017etale}.
\end{enumerate}
\end{example}

\begin{definition}[{\cite[Definition 2.11]{LeBrasresult}}]
Let $\mathcal{BC} \subseteq \mathcal{S}\mathrm{hv}_{\text{pro{\'e}t}}(C; \mathbb{Q}_p)$ (resp. $\mathcal{BC}' \subseteq \mathcal{S}\mathrm{hv}_{v}(C; \mathbb{Q}_p)$) denote the smallest abelian subcategory that is stable under extensions and contains the sheaves $\underline{\mathbb{Q}_p}$ and $\mathbb{G}_a$.
\end{definition}

Note that Banach--Colmez spaces were originally defined on sympathetic algebras (see \cite{BanachColmezspaces}). Indeed, a \emph{sympathetic $C$-algebra} $\Lambda$ is a Banach $C$-algebra equipped with the spectral norm $\Vert \cdot \Vert$ such that $\lambda \mapsto \lambda^p$ is surjective on $\{\lambda \mid \Vert \lambda -1 \Vert <1\}$ (\cite[\S 5]{BanachColmezspaces}). By definition, sympathetic algebras are perfectoids. A Banach--Colmez space is a functor
from the category of sympathetic $C$-algebras to the category of $\mathbb{Q}_p$-Banach spaces. A morphism of Banach--Colmez spaces is a natural transformation of functors. These Banach--Colmez spaces form a category. 

By \cite[Proposition 7.11 \& Remark 2.12]{LeBrasresult}, we know that this category is equivalent to $\mathcal{BC}$ and $\mathcal{BC}'$. There are also notions of dimension and height on Banach--Colmez spaces, and these are additive on the short exact sequences of Banach--Colmez spaces (see {\cite[\S 7]{BanachColmezspaces}} for more details).

Fontaine's constructions of period rings (e.g. $B_{\mathrm{cris}}^+$ and $B_{\mathrm{dR}}^+$), 
extend functorially without any issues if we replace $C$ by any sympathetic algebra $\Lambda$ (and even by any Banach algebra); see \cite[\S  8 \& \S 9]{BanachColmezspaces}. Consequently, we obtain period sheaves
$\mathbb{B}_{\mathrm{cris}}^+$ and $\mathbb{B}_{\mathrm{dR}}^+$
(which are viewed as functors from the category of sympathetic $C$-algebras to the category of usual $\mathbb{Q}_p$-Banach algebras in this paper) whose $C$-values are the usual period rings: $\mathbb{B}_{\mathrm{cris}}^+(C) = B_{\mathrm{cris}}^+$ and $\mathbb{B}_{\mathrm{dR}}^+(C) = B_{\mathrm{dR}}^+$. 
From these Banach algebras, 
one can obtain important examples of Banach--Colmez spaces as follows:

\begin{example}
\label{eg:BC:period.rings}
\begin{enumerate}
	\item 
	For any integer $m \ge 0$, let $\mathbb{B}_m = \mathbb{B}_{\mathrm{dR}}^+/\mathrm{Fil}^m  \mathbb{B}_{\mathrm{dR}}^+$, here $\mathrm{Fil}^m  \mathbb{B}_{\mathrm{dR}}^+= t^m \mathbb{B}_{\mathrm{dR}}^+$ and $t$ is the $p$-adic analog of $2\pi i$. Then $\mathbb{B}_m$ is a Banach--Colmez space of Dimension
		$$\operatorname{Dim} \mathbb{B}_m = (m, 0);$$
	 see \cite[Corollaire 9.23]{BanachColmezspaces}. 
	\item 
	For any pair $(d,h) \in \mathbb{N} \times \mathbb{Z}_{>0}$, let $\mathbb{U}_{d,h} = (\mathbb{B}_{\mathrm{cris}}^+)^{\varphi^h=p^d}$. Then $\mathbb{U}_{d,h}$ is a Banach--Colmez space of Dimension
		$$ \operatorname{Dim} \mathbb{U}_{d,h}  = (d, h);$$
see \cite[Proposition 1.6]{BanachColmezspaces}.
	Notice that $\mathbb{U}_{0,h} = \underline{\mathbb{Q}_{p^h}}$ (\cite[Proposition 9.15]{BanachColmezspaces}), where $\mathbb{Q}_{p^h}$ is the unique unramified extension of $\mathbb{Q}_p$ with residue field $\mathbb{F}_{p^h}$. 
	\item For any positive integers $d$ and $h$, let $\mathbb{V}_{d,-h} = \mathbb{B}_{\mathrm{dR}}^+/(\mathrm{Fil}^d(\mathbb{B}_{\mathrm{dR}}^+) + \underline{\mathbb{Q}_{p^h}}) \simeq \mathbb{B}_d/\underline{\mathbb{Q}_{p^h}}$, where $\mathbb{Q}_{p^h}$ is the unramified extension of $\mathbb{Q}_p$ of degree $h$.
    Then $\mathbb{V}_{d,-h}$ is a Banach--Colmez space of Dimension: 
	$$\operatorname{Dim} \mathbb{V}_{d,-h} = (d, -h).$$
This follows from \cite[\S 7]{BanachColmezspaces}.
\end{enumerate}
\end{example}

\begin{definition}\label{defn:dimension vector}
    For any pair of integers $(d,h)\in (\mathbb{N}\times \mathbb{Z})\backslash (\{0\}\times \mathbb{Z}_{\leq 0})$, we use $\mathbb{E}_{d,h}$ to denote the Banach--Colmez space in the following way.$$\mathbb{E}_{d,h}\coloneqq \begin{cases} \mathbb{B}_d, & \text{if $h=0$}; \\ \mathbb{U}_{d,h}, & \text{if $h>0$}; \\ \mathbb{V}_{d,h}, & \text{if $h<0$}.
    \end{cases}$$
\end{definition}

\subsection{Fargues--Fontaine curves} The category $\mathcal{BC}$ of Banach-Colmez spaces is closely related to the
category of coherent sheaves on the absolute Fargues--Fontaine curve. Recall that given $C^\flat$ a complete algebraically closed non-Archimedean field over $\bar{\mathbb{F}}_p$, a complete discrete valuation field $E$ with perfect residue field, and $\pi$ being a uniformizer of $E$, there is a curve $X_{C^\flat,E}$ defined as the projective scheme associated to the graded algebra $P_{C^\flat,E}$ with 
\[
P_{C^\flat,E}\coloneqq \oplus_{i\geq 0} ({B}_{\mathrm{cris},E}^+)^{\varphi=\pi^i}.
\]
In particular, if we let $E=\mathbb{Q}_p$, then $X_{FF}=X_{C^\flat,\mathbb{Q}_p}$ is the absolute Fargues--Fontaine curve.

There exist functions of rank and degree on the abelian category $\mathrm{Coh}(X_{FF})$, the associated notion of stability, and Harder--Narasimhan filtrations (see \cite[\S 5]{farguesfontaine-courbes} for these facts).

By the GAGA equivalence $\mathrm{Bun}_{\mathcal{X}_{C^\flat}} \simeq \mathrm{Bun}_{X_{FF}}$ (see \cite[Theorem 6.3.9]{kedlaya2015relative} and \cite[Proposition II.2.7]{fargues2021geometrization}), the fact that every bounded coherent complex on $X_{FF}$ is a perfect complex, and Scholze's equivalence $\mathrm{Perf}_{C;v} \simeq \mathrm{Perf}_{C^\flat, v}$ (implying $D_{v}(C; \mathbb{Q}_p) \simeq D_{v}(C^\flat; \mathbb{Q}_p)$, see \cite{scholze2017etale}  or \cite[Theorem 2.7]{LeBrasresult}), we obtain a functor 
	$$R \tau_* \colon D^b(\mathrm{Coh}(X_{FF})) \to D_{v}(C; \mathbb{Q}_p).$$
In particular, for any bounded coherent complex $\mathcal{E} \in  D^b(\mathrm{Coh}(X_{FF}))$ and $i \in \mathbb{Z}$, $R^i \tau_*(\mathcal{E})$ is a $v$-sheaf of $\mathbb{Q}_p$-vector spaces on $\mathrm{Perf}_C$.

For the relation between Banach--Colmez spaces and the absolute Fargues--Fontaine curve $X_{FF}$, one needs to consider a tilted $t$-structure of the derived category of coherent sheaves on $X_{FF}$. In fact, the existence of Harder--Narasimhan filtrations on $\mathrm{Coh}(X_{FF})$ provides us with a new tilted $t$-structure in $D^b(X_{FF})\coloneqq D^b(\mathrm{Coh}(X_{FF}))$ using the standard tilting method on torsion pairs (see \cite[Proposition 2.1]{happel1996tilting} for more details of the tilting method).

Hence, we have a full subcategory $\mathrm{Coh}^-(X_{FF})\subset D^b(X_{FF})$ defined in the following way:
$$\mathrm{Coh}^-(X_{FF})\coloneqq \{\mathcal{E}\in D(X_{FF}), H^i(\mathcal{E})=0\ \text{for $i\neq -1,0$, $H^{-1}(\mathcal{E})\in \mathcal{F}$, $H^0(\mathcal{E})\in\mathcal{T}$}\},$$ where $$\mathcal{T}\coloneqq \{E\in \Coh(X_{FF})\mid \mu_{min}(E)>0\}, $$ and $$\mathcal{F}\coloneqq \{E\in \Coh(X_{FF})\mid \mu_{max}(E)\leq 0  \}$$ are two full subcategories in $\Coh(X_{FF}).$ Here $\mu_{max}(E), \mu_{min}(E)$ denote the slopes of the first and the last Harder--Narasimhan factors of $E$ respectively.

\begin{theorem}[{\cite[Théorème 1.2]{LeBrasresult}}]\label{Thm: Le bras' thm}
    The functor $R^0\tau_*$ induces an exact equivalence between the abelian categories $\mathrm{Coh}^-(X_{FF})$ and $\mathcal{BC}$.
\end{theorem}

The following equivalence between derived categories was implicit in Le Bras' paper \cite{LeBrasresult}, and follows from the above equivalence and cotilting theory.

\begin{prop}
    The two bounded derived categories $D^b(\mathcal{BC})$ and $D^b(X_{FF})$ are equivalent to each other.
\end{prop}
\begin{proof}
 The abelian category $\mathcal{BC}$ is the tilting category of $\mathrm{Coh}(X_{FF})$ with respect to the torsion pair $(\mathcal{T}, \mathcal{F})$ in the sense of \cite{happel1996tilting}.
 
 According to \cite{happel1996tilting} (see also \cite[Prop. 5.4.3]{alexey2003generators}), to prove the equivalence, it suffices to show that the above torsion pair $(\mathcal{T}, \mathcal{F}) $ is \emph{cotilting} in the sense that every object in $\mathrm{Coh}(X_{FF})$ is a quotient of an object in $\mathcal{F}$. In fact, every torsion sheaf on $X_{FF}$ is a quotient of $\mathcal{O}_{X_{FF}}^{\oplus m}$ for some $m > 0$, and thus a quotient of $\mathcal{O}_{X_{FF}}(-n)^{\oplus m} \in \mathcal{F}$ for some sufficiently large $n$. By \cite[Prop. 6.2.4]{kedlaya2015relative} (see also \cite[Thm. II.2.6]{fargues2021geometrization}), every vector bundle on $X_{FF}$ can be obtained as a quotient of $\mathcal{O}_{X_{FF}}(-n)^{\oplus m} \in \mathcal{F}$ for some $m>0$ and sufficiently large $n$. Therefore, the torsion pair $(\mathcal{T}, \mathcal{F})$ is cotilting, and the equivalence follows.
\end{proof}
\begin{remark}
    Note that in general, the derived category of a tilted heart is not necessarily equivalent to the original triangulated category it lies in (see \cite[Exercise 5.3]{macri2017lectures}).
\end{remark}

\begin{remark}\label{rmk:corresponding vector bundles}
    By this equivalence, the examples in \ref{eg:BC:period.rings} correspond to the following objects in $D(X_{FF})$:
    \begin{enumerate}
        \item For any integer $m \ge 0$, the object corresponding to $\mathbb{B}_{m}$ is the indecomposable torsion sheaf $\mathcal{T}_m$ of length $m$ supported at $\infty$ in $X_{FF}$ corresponding to the untilt $C$ of $C^\flat$.
        \item For any pair $(d,h) \in \mathbb{N} \times \mathbb{Z}_{>0}$, the object corresponding to $\mathbb{U}_{d,h}$ is the stable vector bundle $\mathcal{O}(\frac{d}{h})$ of slope $\frac{d}{h}$ (see \cite[\S 5.3]{farguesfontaine-courbes} for the construction of $\mathcal{O}(\frac{d}{h})$). 
        \item For any positive integers $d$ and $h$, the object corresponding to $\mathbb{V}_{d,-h}$ is the object $\mathcal{O}(\frac{-d}{h})[1],$ where $\mathcal{O}(\frac{-d}{h})$ is the stable vector bundle on $X_{FF}$ with degree $-d$ and rank $h$. 
    \end{enumerate}
\end{remark}



\section{Modular variants of fundamental lemma}

Let us recall Colmez's version of fundamental lemma.

\begin{theorem}[{\cite[Proposition 7.11]{BanachColmezspaces}}]\label{theorem:the fundamental lemma}
   Let \(\mathbb{W}\) be a Banach--Colmez space admitting a presentation of the form
\[
0 \longrightarrow U \longrightarrow \mathbb{W} \xrightarrow{\alpha } \mathbb{V}^1
\longrightarrow 0,
\]
where \(U\) is a finite-dimensional \(\mathbb{Q}_p\)-Banach--Colmez Space. Let
\[
\psi: \mathbb{W} \longrightarrow \mathbb{V}^1
\]
be a morphism of Banach--Colmez spaces. Then one of the following two alternatives holds.
\begin{enumerate}
    \item Either \(\psi\) is surjective, and \(\ker \psi\) is a finite-dimensional
\(\mathbb{Q}_p\)-Banach--Colmez Space whose \(\mathbb{Q}_p\)-dimension is $\dim_{\mathbb{Q}_p} U$;

\item or \(\operatorname{Im}\psi\) is a finite-dimensional \(\mathbb{Q}_p\)-Banach--Colmez Space, and
\[
0 \longrightarrow U\cap \ker\psi
\longrightarrow \ker\psi
\xrightarrow{\alpha} \mathbb{V}^1
\longrightarrow 0
\]
is a presentation of \(\ker\psi\), with
\[
\dim_{\mathbb{Q}_p}(U\cap \ker\psi)
=
\dim_{\mathbb{Q}_p}U-\dim_{\mathbb{Q}_p}(\operatorname{Im}\psi).
\]
\end{enumerate}
\end{theorem}

Here $\mathbb{V}^1$ is $\mathbb{E}_{1,0}$ in our notation. This result follows directly from the following statements. \begin{prop}\label{prop: real statement}
Under the same assumptions of Theorem \ref{theorem:the fundamental lemma}, the following statements are true.\begin{enumerate}
\item for any nontrivial morphism $\phi:\mathbb{E}_{0,1}\rightarrow \mathbb{E}_{1,0},$ we have that $\phi$ is injective and $\coker(\phi)\simeq  \mathbb{E}_{1,-1}.$ 
   
    \item for any $n\in\mathbb{N}$ and any nontrivial morphism $\psi: \mathbb{E}_{1,n}\rightarrow \mathbb{E}_{1,0},$ we have that $\psi$ is surjective and $\ker\psi$ is a finite-dimensional $\mathbb{Q}_p$-Banach--Colmez space with dimension equal to $n$.
     \item we have that $\mathbb{W}\simeq \mathbb{E}_{1,n}\oplus V$, where $n\in \mathbb{N}$, and $V$ is a finite-dimensional $\mathbb{Q}_p$-Banach--Colmez space with dimension equal to $\dim_{\mathbb{Q}_p}(U)-n.$
    
\end{enumerate}
    
\end{prop}

\begin{remark}
    Although the fundamental lemma does not need the $\coker$ part in Proposition \ref{prop: real statement}.(1), we include it for the symmetry of statements.
\end{remark}

In fact, for any Farey triangle, we can prove a similar result.  Let us recall what a Farey triangle is.

\begin{definition}
\label{def:Farey-triangle}
     A triple 
     $$(\frac{h_1}{d_1},\frac{h_2}{d_2},\frac{h_3}{d_3})$$
     of elements of $\mathbb{Q}_{\infty}$ is called a \textit{Farey triangle} if 
     each rational number is written in simplest form with a non-negative denominator, while
$\infty$ is written as $1/0$, and if
     the following conditions are met: \begin{enumerate}
        \item we have $\frac{h_1}{d_1}>\frac{h_2}{d_2}>\frac{h_3}{d_3}$, and $h_2=h_1+h_3, \ d_2=d_1+d_3;$
        \item $h_1d_2-h_2d_1=h_2d_3-h_3d_2=1.$
    \end{enumerate}
\end{definition}
\begin{remark}
    The extended modular group $\widetilde{\Gamma}=\mathrm{PGL(2,\mathbb{Z})}$ acts on these triangles transitively (see \cite[\S 3]{geometryofcontinuedfractions}). The triangle $(\frac{1}{0},\frac{1}{1},\frac{0}{1})$ is usually called the fundamental Farey triangle. See Figure 1 for a pictorial illustration.
\end{remark}
\begin{remark}
\label{rmk:Farey-triangle-convention}

Note that we write the entries of a Farey triangle as $\frac{h_i}{d_i}$, rather than
as $\frac{d_i}{h_i}$, which is the usual slope of the corresponding vector bundle
$\mathcal{O}(\frac{d_i}{h_i})$ in Remark~\ref{rmk:corresponding vector bundles}. 

We use the convention $\frac{h_i}{d_i}$ for two reasons. First, it is natural
from the point of view of the tilt in Le Bras' equivalence: with this
convention, the ordering on $\mathbb{Q}_\infty$ corresponds to the opposite of
the slope ordering in the tilted heart $\mathrm{Coh}^-(X_{FF})$. Second, when
$\frac{h_i}{d_i}$ is written in lowest terms with $d_i>0$, the denominator
$d_i$ matches the dimension coordinate in the Dimension vector of a
Banach--Colmez space as in  Definition~\ref{defn:dimension vector}.
\end{remark}


For any Farey triangle $(\frac{h_1}{d_1},\frac{h_2}{d_2},\frac{h_3}{d_3})$, we have the following result.

\begin{lemma}\label{lemma:minimal triangle lemma}
    Let $(\frac{h_1}{d_1},\frac{h_2}{d_2},\frac{h_3}{d_3})$ be any Farey triangle, then we have the following results.
    \begin{enumerate}
        \item if $f:\mathbb{E}_{d_1,h_1}\xrightarrow{}\mathbb{E}_{d_2,h_2}$ is a nonzero morphism, then it is injective, and $$\mathrm{coker}(f)\simeq\begin{cases}
            
        \mathbb{E}_{d_3,h_3} & \text{if $h_3\neq 0$;}\\ R^0\tau_*(k_x), &\text{if $h_3= 0$,}\end{cases} $$ where $k_x$ is the skyscraper sheaf of a closed point $x\in X_{FF}$.
        
        \item if $g:\mathbb{E}_{d_2,h_2}\xrightarrow{}\mathbb{E}_{d_3,h_3}$ is a nonzero morphism, then it is surjective and $$\mathrm{ker}(g)\simeq \begin{cases}
            
        \mathbb{E}_{d_1,h_1} & \text{if $h_1\neq 0$}; \\ R^0\tau_*(k_x), &\text{if $h_1= 0$}.\end{cases}$$
        \item if $h:\mathbb{E}_{d_3,h_3}\xrightarrow{}\mathbb{E}_{d_1,h_1}[1]$ is a nonzero morphism, then  $$\mathrm{cone}(h)\simeq \begin{cases}
            
        \mathbb{E}_{d_2,h_2}[1] & \text{if $h_2\neq 0$}; \\ R^0\tau_*(k_x)[1], &\text{if $h_2= 0$}.\end{cases}$$
    \end{enumerate}
\end{lemma}
\begin{proof}
    The proofs of these statements are similar, so we only prove (1). Furthermore, for simplicity of the proof, we deal with the case where $h_1,h_2,h_3$ are positive, as the other cases can be treated similarly. By Theorem \ref{Thm: Le bras' thm} and Remark \ref{rmk:corresponding vector bundles}, it suffices to prove that for any nontrivial morphism $f:\mathcal{O}(\frac{d_1}{h_1})\xrightarrow{} \mathcal{O}(\frac{d_2}{h_2})$ in $\mathrm{Coh}(X_{FF})$, it must be injective and $\coker(f)\simeq \mathcal{O}(\frac{d_3}{h_3}).$   

    In fact, for any nonzero morphism $f$, consider its image $\mathrm{im}(f)\subset \mathcal{O}(\frac{d_2}{h_2}).$ Since it is a quotient object of $\mathcal{O}(\frac{d_1}{h_1})$ and a subobject of $\mathcal{O}(\frac{d_2}{h_2})$, we have $\frac{d_1}{h_1}\leq \mu(\mathrm{im}(f))\leq \frac{d_2}{h_2}$ (because of the stability of $\mathcal{O}(\frac{d_1}{h_1})$ and $\mathcal{O}(\frac{d_2}{h_2})$, and $\rank(\mathrm{im}(f))\leq h_1$). As $d_2h_1-d_1h_2=1$, by Pick's theorem, there is no integral point in the interior of the triangle $(0,0),(d_1,h_1),(d_2,h_2)$ in the complex plane. Hence, we have $\mu(\mathrm{im}(f))=\frac{d_1}{h_1}.$ Therefore, $\ker(f)$ is an object with rank and degree equal to $0$, equivalently $\ker(f)\simeq 0$ and $f$ is injective.

    For $\coker(f)$, consider its last Harder--Narasimhan factor and denote this factor by $Q$. By the definition of Harder--Narasimhan filtration, we know that $\mu(Q)\leq \frac{d_3}{h_3}$. As $Q$ is a quotient object of $\coker(f)$, hence a proper quotient object of $\mathcal{O}(\frac{d_2}{h_2})$, this implies that $\mu(Q)>\frac{d_2}{h_2}$ and $\rank(Q)<h_3$. However, there is no integral point in the interior of the triangle $(0,0),(d_2,h_2),(d_3,h_3)$ in the plane. Hence, we have $\mu(Q)=\frac{d_3}{h_3}$ and $\coker(f)\simeq Q$ is semistable, therefore $Q$ is stable by the primitivity of $(d_3,h_3),$ and isomorphic to $\mathcal{O}(\frac{d_3}{h_3})$. This completes the proof. 
\end{proof}

\begin{remark}
 Taking the Farey triangle to be
$\left(\frac{1}{0},\frac{1}{1},\frac{0}{1}\right)$, we obtain
Proposition \ref{prop: real statement}(3) by inductively applying
Lemma \ref{lemma:minimal triangle lemma}(3). Applying the same argument
to the adjacent Farey triangle
$\left(\frac{1}{0},\frac{0}{1},\frac{-1}{1}\right)$ gives
Proposition \ref{prop: real statement}(1). It remains to prove
Proposition \ref{prop: real statement}(2), for which we need the following result.
\end{remark}

By Lemma \ref{lemma:minimal triangle lemma}, we can prove the following result, which gives us a variant of $p$-adic fundamental lemma for any given Farey triangle.

\begin{theorem}\label{thm:SL2 variants}
    For a given Farey triangle $(\frac{h_1}{d_1},\frac{h_2}{d_2},\frac{h_3}{d_3})$ and a non-negative integer $k\in\mathbb{N}$, we have the following results.

    \begin{enumerate}
        \item if $f:\mathbb{E}_{d_1,h_1}\xrightarrow{}\mathbb{E}_{d_2+kd_3,h_2+kh_3}$ is a nonzero morphism, then it is injective and $$\mathrm{coker}(f)\simeq \begin{cases}
            
       \mathbb{E}_{d_3,h_3}^{\bigoplus k+1} & \text{if $h_3\neq 0$};\\ R^0\tau_*(T)& \text{if $h_3=0$}, \end{cases}$$ where $T$ is a torsion sheaf of length $k+1$ on $X_{FF}$.
        \item  if $g:\mathbb{E}_{d_2+kd_1,h_2+kh_1}\xrightarrow{}\mathbb{E}_{d_3,h_3}$ is a nonzero morphism, then it is surjective and $$\mathrm{ker}(g)\simeq \begin{cases}
            
        \mathbb{E}_{d_1,h_1}^{\bigoplus k+1} & \text{if $h_1\neq 0$}; \\ R^0\tau_*(T), &\text{if $h_1= 0$},\end{cases}$$ where $T$ is a torsion sheaf of length $k+1$ on $X_{FF}$.

        \item for any pair of non-negative integers $k_1,k_3\in\mathbb{N}$, let $$t:\mathbb{E}_{d_3+k_3d_2,h_3+k_3 h_2}\xrightarrow{}\mathbb{E}_{d_1+k_1d_2,h_1+k_1h_2}[1]$$ be a nonzero morphism, then  $$\mathrm{cone}(t)\simeq \begin{cases}
            
        \mathbb{E}_{d_2,h_2}[1]^{\bigoplus k_1+k_3+1} & \text{if $h_2\neq 0$}; \\ R^0\tau_*(T)[1], &\text{if $h_2= 0$},\end{cases}$$ where $T$ is a torsion sheaf of length $k_1+k_3+1$ on $X_{FF}$.

    \end{enumerate}
\end{theorem}
\begin{proof}
   The proof is essentially the same as the proof of Lemma \ref{lemma:minimal triangle lemma}. We include the proof of (3) for the reader's convenience.

   We only deal with the case where $h_1+k_1h_2>0$, $h_3+k_3h_2<0$ as other cases can be proved similarly. By Remark \ref{rmk:corresponding vector bundles}, we know that the nontrivial morphism $$t:\mathbb{E}_{d_3+k_3d_2,h_3+k_3 h_2}\xrightarrow{}\mathbb{E}_{d_1+k_1d_2,h_1+k_1h_2}[1]$$ corresponds to a morphism $\tilde{t}[1]$, where $$\tilde{t}:\mathcal{O}(\frac{-d_3-k_3d_2}{-h_3-k_3h_2})\xrightarrow{} \mathcal{O}(\frac{d_1+k_1d_2}{h_1+k_1h_2}) $$  is a morphism in the category of coherent sheaves on $X_{FF}$. Without loss of generality, one can assume that $h_2\geq 0;$ then the triangle $$(0,0),(-d_3-k_3d_2,-h_3-k_3h_2),(d_1+k_1d_2,h_1+k_1h_2)$$ in the plane is of area $\frac{k_1+k_3+1}{2}$. By Pick's theorem, this triangle has no integral points in its interior, and the only integral points on the boundary of this triangle consist of $$(-d_3+kd_2,-h_3+kh_2), \text{where $k\in\mathbb{Z}$ and $-k_3< k\leq k_1$} $$ in addition to its vertices.

   Using the same argument as in the proof of Lemma \ref{lemma:minimal triangle lemma}, one can show that $\tilde{t}$ must be injective as a morphism in $\Coh(X_{FF})$, and $\coker(\tilde{t})$ is a semistable coherent sheaf of slope $\frac{d_2}{h_2}$ if $h_2>0$ or of slope $\infty$ if $h_2=0$. Hence, we have $$\coker(\tilde{t})=\begin{cases}
       \mathcal{O}(\frac{d_2}{h_2})^{k_1+k_3+1}, & \text{if $h_2>0$}; \\ T, &\text{if $h_2=0$,} \end{cases}$$ where $T$ is a torsion sheaf of length $k_1+k_3+1$ on $X_{FF}$. We have the following distinguished triangle $$\mathcal{O}(\frac{d_1+k_1d_2}{h_1+k_1h_2})\xrightarrow{} \coker(\tilde{t})\rightarrow \mathcal{O}(\frac{-d_3-k_3d_2}{-h_3-k_3h_2})[1]\xrightarrow[]{\tilde{t}[1]}\mathcal{O}(\frac{d_1+k_1d_2}{h_1+k_1h_2})[1],$$ in the derived category $D(X_{FF}).$ This distinguished triangle is a short exact sequence $$0\rightarrow \mathcal{O}(\frac{d_1+k_1d_2}{h_1+k_1h_2})\xrightarrow{} \coker(\tilde{t})\rightarrow \mathcal{O}(\frac{-d_3-k_3d_2}{-h_3-k_3h_2})[1]\rightarrow 0$$ in $\Coh^-(X_{FF})$, and its extension class is $\tilde{t}[1]$. After applying $R^0\tau_*$, we get a short exact sequence in $\mathcal{BC}$, and its extension class corresponds to $t$ as $D^b(X_{FF})\simeq D^b(\mathcal{BC}).$ This completes the proof.
\end{proof}

\begin{remark}
    If we take the Farey triangle to be the fundamental triangle $(\frac{1}{0},\frac{1}{1},\frac{0}{1})$, Theorem \ref{thm:SL2 variants}(2) is Proposition \ref{prop: real statement}(2).
\end{remark}

\bibliographystyle{alpha}
	\bibliography{bibfile}
\end{document}